\documentclass[11pt]{article}
\usepackage{amssymb}
\newtheorem{theorem}{Theorem}
\newtheorem{corollary}[theorem]{Corollary}
\newtheorem{definition}[theorem]{Definition}

\newtheorem{lemma}[theorem]{Lemma}

\newenvironment{proof}[1][Proof]{\textbf{#1.} }{\
\rule{0.5em}{0.5em}}
\parskip \baselineskip
\newcommand{\ds}{\displaystyle}
\newcommand{\R}{\mathbb{R}}

\newcommand{\Z}{\mathbb{Z}}

\newcommand{\ol}{\overline}

\title{Calculating the output distribution of stack filters that are erosion-dilation cascades, in particular $LULU$-filters}

\author{R. Anguelov$^1$, P.W. Butler$^2$, C.H. Rohwer$^2$, M.
Wild$^2$\\[3pt]
$^1$Department of Mathematics and Applied Mathematics, University of
Pretoria
\\
Pretoria 0002, South Africa\\[3pt]
$^2$Department of Mathematical Sciences,
University of Stellenbosch\\
 Private Bag X1, Matieland 7602, South
Africa}
\date{}
\begin{document}
\maketitle

\begin{abstract}
Two procedures to compute the output distribution $\phi_S$ of certain stack filters $S$ (so called erosion-dilation cascades) are given. One rests on the disjunctive normal form of $S$ and also yields the rank selection probabilities. The other is based on inclusion-exclusion and e.g. yields $\phi_S$ for some important $LULU$-operators $S$. Properties of $\phi_S$ can be used to characterize smoothing properties of $S$.

%Given a random sequence $X$ generated independently and identically by a distribution $F_X$ and an operator $S$ it is of importance to compute the distribution $F_{SX}$ of the sequence $SX$. We derive a computation procedure for obtaining a function $\phi_S:[0,1]\to[0,1]$ such that the output distribution $F_{SX}$ for the most commonly used operators of the LULU class are obtained in the form $F_{SX}=\phi_S\circ F_X$. The properties of the distribution transfer function $\phi_S$, which turns out to be a polynomial, are then used to characterize the smoothing properties of the operator $S$.

%Given a random sequence $x$ generated identically independently by a distribution $f$, it is of importance to compute the distribution $f_s$ of the sequence $S_x$ for $S$ an operator from the class of $LULU$-smoothers. A computational procedure for obtaining the function $f_s$ from $f$ is derived, which is much simpler/faster than previous more general methods.
\end{abstract}

\section{Introduction}

The LULU operators are well known in the nonlinear multiresolution
analysis of sequences. The notation for the basic operators
$L_n$ and $U_n$, where $n\in\mathbb{N}$ is a parameter related to
the window size, has given rise to the name LULU for the theory of
these operators and their compositions. Since the time they were
introduced nearly thirty year ago, while also being used in
practical problems, they slowly led to the development of a
new framework for characterizing, evaluating, comparing and
designing nonlinear smoothers. This framework is based on concepts
like idempotency, co-idempotency, trend preservation, total variation preservation,
consistent decomposition. 

As opposed to the deterministic nature of the above
properties, the focus of this paper is on properties of the LULU
operators in the setting of random sequences. More precisely, this
setting can be described as follows: Suppose that $X$ is a
bi-infinite sequence of random variables $X_i$ $(i\in\mathbb{Z})$
which are independent and with a common (cumulative) distribution
function $F_X(t)$ from $L^1([0,1],[0,1])$.
Let $S$ be a smoother. Then we consider the following two
questions:

1. Find a map $\phi_S:[0,1]\to[0,1]$ such that the common {\it output distribution} $F_{SX}(t)$ of $(SX)_i$ ($i\in\mathbb{Z}$) equals $F_{SX} = \phi_S \circ F_X$. The
function $\phi_S$ is also called \emph{distribution transfer}. 

2. Characterize the smoothing effect an operator $S$ has on a
random sequence $X$ in terms of the properties of the common
distribution of $(SX)_i$ ($i\in\mathbb{Z}$).

With regard to the first question we present a new technique which
one may call "expansion calculus" which uses a shorthand notation
for the probability of composite events and a set of rules for
manipulation. Using this technique we provide new elegant proofs
of the earlier results in \cite{ButlerMSc} for the distribution
transfer of the operators $L_nU_n$ and (dually) $U_nL_n$. The power of this
approach is further demonstrated by deriving the distribution
transfer maps for the alternating sequential filters
$C_n= L_nU_nL_{n-1}U_{n-1}...L_1U_1$ and $F_n=U_nL_nU_{n-1}L_{n-1}...U_1L_1$.

With regard to the second question, we may note that it is
reasonable to expect that a smoother should reduce the standard
deviation of a random sequence. Indeed, for simple distributions
(e.g. uniform) and filters with small window size (three point
average, $M_1$, $L_1U_1$, $U_1L_1$) when the computations can be
carried out a significant reduction of the standard deviation is
observed (for the uniform distribution the mentioned filters
reduce the standard deviation respectively by factors of 3, 5/3,
1.293, 1.293). However, in general, obtaining such results is to a
large extent practically impossible due to the technical
complexity particularly when nonlinear filters are concerned. In
this paper we propose a new\footnote{Related concepts of robustness exist. We touch upon one of them in Section 6.} concept of {\it robustness} which characterizes the
probability of the occurrence of outliers rather then considering
the standard deviation.  Upper
robustness characterizes the probability of positive outliers
while the lower robustness characterizes the probability of
negative outliers. In general, the higher the order of robustness
of a smoother the lower the probability of occurrence of outliers
in the output sequence. In terms of this concept it is easy to
characterize a smoother given its distribution transfer function.

The paper is structured as follows. In the next section we give
the definitions of the LULU operators with some fundamental
properties. The concept of robustness is defined and studied in
Section 3. In Section 4 we show how the inclusion-exclusion principle helps to obtain the distribution transfer function of erosion-dilation cascades, a kind of operator frequently used in Mathematical Morphology. This method is considerably refined in Section 5 where it is applied to LULU-operators. They are particular cases of such cascades. Formulas for the major LULU-operators are obtained explicitly or recursively. Using these results the robustness of these operators is also analyzed. Section 6 proposes to substitute inclusion-exclusion by some novel {\it principle of exclusion} which excels for erosion-dilation cascades that don't allow the refinements of inclusion-exclusion possible for $LULU$-operators.

\section{The basics of the $LULU$ theory}

Given a bi-infinite sequence $x=(x_i)_{i\in\mathbb{Z}}$ and
$n\in\mathbb{N}$ the basic LULU operators $L_n$ and $U_n$ are
defined as follows
\begin{eqnarray}
(L_nx)_i&:=&(x_{i-n} \wedge x_{i-n+1} \wedge \cdots \wedge x_i)
\vee (x_{i-n+1} \wedge \cdots \wedge x_{i+1}) \vee \cdots \vee
(x_i \wedge \cdots \wedge x_{i+n})\label{Ln}\\
(U_nx)_i&:=&(x_{i-n}\vee x_{i-n+1} \vee \cdots \vee x_i) \wedge
(x_{i-n+1} \vee \cdots \vee x_{i+1}) \wedge \cdots \wedge (x_i\vee
\cdots \vee x_{i+n})\label{Un}
\end{eqnarray}
 where $\alpha \wedge \beta : = \min
(\alpha, \beta)$, and $\alpha \vee \beta := \max (\alpha, \beta)$
for all $\alpha, \beta \in \R$. Central to the theory is the
concept of {\bf separator}, which we define below. For every
$a\in\mathbb{Z}$ the operator $E_a:\mathbb{R}^\mathbb{Z}\to
\mathbb{R}^\mathbb{Z}$ given by
\[
(E_ax)_i=x_{i+a}, \quad  i\in\mathbb{Z},
\]
is called a {\bf shift operator}.
\begin{definition}\label{defSep}
An operator $S:\mathbb{R}^\mathbb{Z}\to \mathbb{R}^\mathbb{Z}$ is
called a {\bf separator} if
\begin{enumerate}
\item[(i)]  $S\circ E_a=E_a\circ S,\ a\in\mathbb{Z}$ ;\hspace{3cm}\mbox{ (horizontal shift invariance)}
\item[(ii)] $P(f+c)=P(f)+c,\ f,c\in \mathbb{R}^\mathbb{Z}$; \hspace*{1cm} $c$-
constant function  (vertical shift invariance);
\item[(iii)] $P(\alpha f)=\alpha P(f),\ \alpha\in\mathbb{R},\ \alpha\geq
0, \ f\in\mathbb{R}^\mathbb{Z}$; \hspace*{.5cm} (scale invariance)
\item[(iv)] $P\circ P=P$; \hspace*{5.4cm} (Idempotence)
\item[(v)] $(id-P)\circ (id-P)=id-P$. \hspace{2.3cm} (Co-idempotence)
\end{enumerate}
\end{definition}
The first two axioms in Definition \ref{defSep} and partially the
third one were first introduced as required properties of
nonlinear smoothers by Mallows, \cite{Mallows}. Rohwer further
made the concept of a smoother more precise by using the
properties (i)--(iii) as a definition of this concept. The axiom
(iv) is an essential requirement for what is called a {\it
morphological filter},  \cite{Serra82}, \cite{Serra88},
\cite{Soille}. In fact, a {\bf morphological filter} is exactly an
increasing operator which satisfies (iv). The co-idempotence axiom
(v) in Definition \ref{defSep} was introduced by Rohwer in
\cite{Rohwerbook}, where it is also shown that it is an essential
requirement for operators extracting signal from a sequence. More
precisely, axioms (iv) and (v) provide for consistent separation
of noise from signal in the following sense: Having extracted a
signal $Sx$ from a sequence $x$, the additive residual $(I-S)x$,
the noise, should contain no signal left, that is $S\circ(I-S)=0$.
Similarly, the signal $Sx$ should contain no noise, that is
$(I-S)\circ S=0$. It was shown in \cite{Rohwerbook} that $L_n$,
$U_n$ and their compositions $L_nU_n, U_nL_n$ are separators.

The smoothing effect of $L_n$ on an input sequence is the removal
of picks, while the smoothing effect of $U_n$ is the removal of
pits. The composite effect of the two $LU$-{\it operators} $L_n U_n$ and $U_n L_n$ is that the output
sequence contains neither picks nor pits which will fit in the
window of the operators. These are the so called $n$-monotone
sequences, \cite{Rohwerbook}. Let us recall that a sequence $x$ is
$n${\bf -monotone} if any subsequence of $n+1$ consecutive elements is
monotone. For various technical reasons the analysis is typically
restricted to the set $\mathcal{M}_1$ of absolutely summable
sequences. Let $\mathcal{M}_n$ denote the set of all sequences
$x\in\mathcal{M}_1$ which are $n$-monotone. Then
\[
\mathcal{M}_n= Range(L_nU_n) = Range (U_nL_n)
\]
is the set of signals.

The power of the $LU$-operators as separators is further
demonstrated by their trend preservation properties. Let us
recall, see \cite{Rohwerbook}, that an operator is called
{\bf neighbor trend preserving} if $(Sx)_i\leq (Sx)_{i+1}$
whenever $x_i\leq x_{i+1}$, $i\in\mathbb{N}$.
An operator $S$ is
{\bf fully trend preserving} if both $S$ and $I-S$ are neighbor
trend preserving. The operators $L_n$, $U_n$ and all their
compositions are fully trend preserving. 
With the total variation of a sequence,
\[
TV(x) \quad = \quad \sum_{i\in\mathbb{N}}|x_i-x_{i+1}|,\ \ x\in \mathcal{M}_1
\]
a generally accepted measure for the amount of contrast 
present, since it is a semi-norm on $\mathcal{M}_1$, any
separation may only increase the total variation. More precisely,
for any operator $S: \mathcal{M}_1 \to \mathcal{M}_1$ we have
\begin{equation}\label{TVinequ} TV(x) \quad  \le \quad TV(Sx) +
TV((I-S)x).
\end{equation}
All operators $S$ that are fully trend preserving have variation preservation, in that
\begin{equation}\label{TVPre} TV(x)
\quad = \quad TV(Sx) + TV((I-S)x).
\end{equation}
We mention these properties because they provide but few of the motivation for studying the robustness of operators, when the popular medians are optimal in that respect. We intend to show that some $LULU$-composition are nearly as good as the medians, but have superiority most important aspects.

An operator $S$ satisfying property (\ref{TVPre}) is called
{\bf total variation preserving}, \cite{RohwerQM2002}. As
mentioned already, the $LU$-operators are total variation
preserving.

\section{Distribution transfer and degree of robustness of a smoother}

Suppose that $X$ is a bi-infinite sequence of random variables
$X_i$ $(i\in\mathbb{Z})$ which are independent and with a common
(cumulative) distribution function $F_X$. Let $S$ be a smoother.
As stated in the introduction we seek a function
$\phi_S:[0,1]\to[0,1]$, called a \emph{distribution transfer function}
such that
\begin{equation}\label{transfer}
F_{SX}=\phi_S\circ F_X
\end{equation}
is the common distribution of $(SX)_i$ ($i\in\mathbb{Z}$). We
should note that for an arbitrary smoother the existence of such a
distribution transfer function is not obvious. However, for the
smoothers typically considered in nonlinear signal processing (i.e. stack filters of which the LULU operators are particular cases) such a function
does not only exist but it is a polynomial. For example, it is
shown in \cite{Nodes} that the distribution transfer function of
the ranked order operators
\[
(R_{nk}x)_i=\mbox{ the $k$th smallest value of
}\{x_{i-n},...,x_{i+n}\}
\]
is given by
\begin{equation}\label{transferRank}
\phi_{R_{nk}}(p)=\sum\limits_{j=k}^{2n+1}{2n+1 \choose j}
p^j(1-p)^{2n+1-j}.
\end{equation}
The popular median smoothers $M_n$, $n\in\mathbb{N}$, are
particular cases of the ranked order operators, namely
$M_n=R_{n,n+1}$. Hence we have
\begin{equation}\label{transferMedian}
\phi_{M_n}(p)=\sum\limits_{j=n+1}^{2n+1}{2n+1 \choose
j}p^j(1-p)^{2n+1-j}.
\end{equation}
Note that in terms of (\ref{transfer}) the common distribution
function of $(M_nX)_i$, $i\in \mathbb{Z}$, is
\[
F_{M_nX}(t)=\sum\limits_{j=n+1}^{2n+1}{2n+1 \choose
j}F^j_X(t)(1-F_X(t))^{2n+1-j},\ t\in\mathbb{R}.
\]
Using that
\begin{equation}\label{MedianTransDer}
\frac{d}{dz}\phi_{M_n}(p)=(2n+1){2n \choose n}p^n(1-p)^n
\end{equation}
its density is
\[
f_{M_nX}(t)=\frac{d}{dz}\phi_{M_n}(F_X(t))f_X(t)=(2n+1){2n \choose
n}F^n_X(t)(1-F_X(t))^nf_X(t)
\]
where $f_X(t)=\frac{d}{dt}F_X(t)$, $t\in\mathbb{R}$, is the common
density of $X_i$, $i\in \mathbb{Z}$. The distribution of the
output sequence of the basic smoothers $L_n$ and $U_n$ is derived
in \cite{Rohwerbook}.
Equivalently these results can be formulated in
terms of distribution transfer. More precisely we have
\begin{eqnarray}
\phi_{L_n}(p)&=&1-(n+1)(1-p)^{n+1}+n(1-p)^{n+2}\label{TransLn}\\
\phi_{U_n}(p)&=&(n+1)p^{n+1}-np^{n+2}\label{TransUn}
\end{eqnarray}

A primary aim of the processing of signals through nonlinear
smoothers is the removal of impulsive noise. Therefore, the power
of such a smoother can be characterized by how well it eliminates
outliers in a random sequence. The concepts of robustness of a
smoother introduced below are aimed at such characterization.

\begin{definition} A smoother
$S:\mathbb{R}^\mathbb{Z}\to\mathbb{R}^\mathbb{Z}$ is called {\bf lower
robust} of order $r$ if there exists a constant $\alpha>0$ such
that for every bi-infinite sequence $X$ of identically distributed
random variables $X_i$ ($i\in\mathbb{Z}$) there exists
$t_0\in\mathbb{R}$ such that $P(X_i<t)<\varepsilon$ implies
$P((SX)_i<t)<\alpha\varepsilon^k$ for all $t<t_0$ and
$\varepsilon>0$.

Similarly, a smoother
$S:\mathbb{R}^\mathbb{Z}\to\mathbb{R}^\mathbb{Z}$ is called {\bf upper
robust} of order $r$ if there exists a constant $\alpha>0$ such
that for every bi-infinity sequence $X$ of identically distributed
random variables $X_i$ ($i\in\mathbb{Z}$) there exists
$t_0\in\mathbb{R}$ such that $P(X_i>t)<\varepsilon$ implies
$P((SX)_i>t)<\alpha\varepsilon^k$ for all $t>t_0$ and
$\varepsilon>0$.

A smoother which is both lower robust of order $r$ and upper
robust of order $r$ is called {\bf robust} of order $k$.
\end{definition}

The reasoning behind these concepts is simple:  If a distribution
density is heavy tailed, there is a probability $\varepsilon$ that
the size of a random variable is excessively large  (larger than
$t$) in absolute value. Using a non-linear smoother we would aim
to restrict this to an acceptable probability
$\alpha\varepsilon^k$ that such an excessive value can appear in
$SX$, by choosing a smoother with the order of robustness $k$.

Clearly there is a general problem of smoothing: a trade-off
to be made between making a smoother more robust, and the
(inevitable) damage to the underlying signal preservation. (A
smoother clearly cannot create information, but only selectively
discard it.) This is fundamental. There are two main reasons for
using one-sided robustness: Firstly, the unreasonable pulses often
are only in one direction, as in the case of "glint" in signals
reflected from objects with pieces of perfect reflectors, and
there clearly are no reflections of negative intensity
possible.  Secondly, we may chose smoothers that are not
symmetric, as are the $LU$-operators, for reasons that are of
primary importance. In this case the robustness is determined from
the sign of the impulse. 

The robustness of a smoother can be characterized through its
distribution transfer function as stated in the theorem below.

\begin{theorem}\label{theoRobust}
Let the smoother $S$ have a distribution transfer function
$\phi_S$. Then
\begin{itemize}\item[a)]$S$ is lower robust of order $r$ if and only if
$\phi_S(p)=O(p^r)$ as $p\to 0$. \item[b)]$S$ is upper robust of
order $r$ if and only if $\phi_S(1-p)-1=O(p^r)$ as $p\to 0$.
\end{itemize}
\end{theorem}
\begin{proof}Points a) and b) are proved using similar arguments. Hence we prove only
a). Let $\phi_S(p)=O(p^r)$ as $p\to 0$. This means that there
exists $\alpha>0$ and $\delta>0$ such that $\phi_S(p)<\alpha
p^k$ for all $p\in [0,\delta)$. Let $X$ be a sequence of
identically distributed random variables with common distribution
function $F_X$. Since $\lim_{t\to -\infty}F_X(t)=0$, there exists $t_0$
such that $F_X(t_0)<\delta$. Let $t<t_0$ and $\varepsilon>0$ be
such that $P(X_i<t)<\varepsilon$. The monotonicity of $F_X$
implies that $F_X(t)\in[0,\delta)$. Then
\[
P((SX)_i<t)=F_{SX}(t)=\phi_S(F_X(t))<\alpha (F_X(t))^k<\alpha
\varepsilon^k,
\]
which proves that $S$ is lower robust of order $k$. It is easy to
see that the argument can be reversed so that the stated condition
is also necessary.
\end{proof}

In the common case when the distribution transfer function is a
polynomial,  conditions a) and b) can be formulated in a much
simpler way as given in the next corollary.

\begin{corollary}\label{corRobust}
Let the distribution transfer function of a smoother $S$ be a
polynomial $\phi_S$ . Then
\begin{itemize}\item[a)]$S$ is lower robust of order $r$ if and only if
$p=0$ is a root of order $r$ of $\phi_S$.

\item[b)]$S$ is upper robust of order $r$ if and only if $p=1$ is
a root of order $r$ of $\phi_S-1$.
\end{itemize}
\end{corollary}

Using the distribution transfer functions given in (\ref{TransLn})
and (\ref{TransUn}) it follows from Corollary \ref{corRobust} that
$U_n$ is lower robust of order $n+1$ and that $L_n$ is upper
robust of order $n+1$.

The robustness of the median filter $M_n$ can be obtained from
(\ref{transferMedian}). Obviously $p=0$ is a root of order $n+1$.
Furthermore, $\phi_{M_n}(1)=1$. Then using also that $p=1$ is a
root of order $n$ of $\frac{d}{dz}\phi_{M_n}$, see
(\ref{MedianTransDer}), we obtain that $p=1$ is a root of order
$n+1$ of $\phi_{M_n}-1$. Therefore, $M_n$ is robust of order
$n+1$.

Clearly with symmetric smoothers, in that $S(-x)=-S(x)$,  the
concepts of lower and upper robustness are not needed, as is the
case for example with $M_n$. However, we have to recall in this
regard that the operators $L_n$, $U_n$ and their compositions,
which are the primary subject of our investigation,  are not
symmetric. A useful feature of the lower and upper robustness is
that it can be induced through the point-wise defined partial order
between the operators. Let us recall that given the maps
$A,B:\mathbb{R}^\mathbb{Z}\to \mathbb{R}^\mathbb{Z}$, the relation
$A\leq B$ means that $Ax\leq Bx$ for all
$x\in\mathbb{R}^\mathbb{Z}$.

\begin{theorem}\label{theoOrder1}Let $A,B:\mathbb{R}^\mathbb{Z}\to \mathbb{R}^\mathbb{Z}$ be smoothers.
If $A \leq B$ then $\phi_B \leq \phi_A$.
\end{theorem}
\begin{proof}
Let $X$ be a sequence of independent random variables $X_i$ ($i\in\mathbb{Z}$) uniformly distributed on $[0,1]$ . Let $p\in[0,1]$. If $t$ is such that $p = F_X(t)$ then
\begin{eqnarray*}
\phi_B(p)&=&\phi_B(F_X(t))\ =\ F_{BX}(t)\ =\ P((BX)_i\leq t)\\
&\leq&P((AX)_i\leq t)\ =\ F_{AX}(t)\ =\ \phi_A(F_X(t))\ =\
\phi_A(p).
\end{eqnarray*}
\end{proof}

As a direct consequence of Theorem \ref{theoOrder1} and Theorem
\ref{theoRobust} we obtain the following theorem.

\begin{theorem}\label{theoOrder2}
Let $A,B:\mathbb{R}^\mathbb{Z}\to \mathbb{R}^\mathbb{Z}$ be smoothers such
that $A\leq B$. Then
\begin{itemize}
\item[a)] If $A$ is lower robust to the order $k$, then so is $B$.

\item[b)] If $B$ is upper robust to the order $k$, then so is $A$.

\end{itemize}
\end{theorem}

Using Theorem \ref{theoOrder2} one can derive statements about
the lower robustness and the upper robustness of the $LU$-operators:
\begin{equation}\label{LnMnUnineq}
U_n L_n\leq M_n\leq L_nU_n
\end{equation}
Therefore, $U_n L_n$ inherits the upper-robustness of $M_n$, while
$L_nU_n$ inherits the lower-robustness of $M_n$. More precisely
\begin{eqnarray}
&\bullet&U_n L_n \mbox{ is upper robust of order }n+1;\hspace{5cm}\label{UnLnupperrobust}\\
&\bullet&L_n U_n \mbox{ is lower robust of order
}n+1 \nonumber %\label{LnUnlowerrobust}
\end{eqnarray}

One may expect that, since $L_n$ is upper robust of order $n+1$ and
$U_n$ is lower robust of order $n+1$, their compositions
should be both lower and upper robust of order $n+1$. However,
as we will see later, this is not the case. The problem is the
following. The definition of robustness requires that the random
variables in the sequence $X$ are identically distributed but they
are not necessarily independent. However, the distribution
transfer functions $\phi_{L_n}$ and $\phi_{U_n}$ are derived under
the assumption of such independence. Noting that entries in the
sequences $L_nX$ are not independent, it becomes clear that the
common distribution of $U_nL_nX$ cannot be obtained by applying
$\phi_{U_n}$ to $F_{L_nX}$. More generally, since the distribution
transfer functions are derived for sequences of independent
identically distributed random variables the equality
$\phi_{AB}=\phi_A\circ\phi_B$ does not hold for arbitrary
operators $A$ and $B$. Therefore the order of robustness of $B$ is
not necessarily preserved by the composition $AB$.

Observe that another concept of robustness is introduced in [10]. Other than Definition 2 it only applies to stack filters. The concept is similar in that it also based on certain probabilities (in this case ``selection probabilities'').

\section{The output distribution of arbitrary erosion-dilation cascades}

Here we present a method for obtaining output distributions of so called erosion-dilation cascades (defined below). It essentially uses the inclusion-exclusion principle for the probability of simultaneous events. For convenience we recall this principle below. For $n =2$ the easy proof will be given along the way.

\begin{lemma}
For any random  variables  $Z_1, Z_2 \cdots Z_n$ it holds that

$\begin{array}{lll} P(Z_1, \cdots, Z_n \leq t) & = & 1-  \ds\sum_{i=1}^n P(Z_i > t) + \ds\sum_{1\leq i < j \leq n} P(Z_i, Z_j > t) \\
\\
& & - \ds\sum_{1 \leq i < j < k \leq n} P (Z_i, Z_j, Z_k > t) \ + \  \cdots \ + \ (-1)^n P(Z_1, \cdots, Z_n > t) \end{array}$
\end{lemma}

Let us recall that in the general setting of Mathematical Morphology \cite{Serra88} the basic operators $L_n$ and $U_n$ are morphological opening and closing respectively. As such they are compositions of an erosion and a dilation. More precisely, for a sequence $x=(x_i)_{i\in\mathbb{Z}}$ we have

\hspace*{4cm} $\begin{array}{lll} (L_nx)_i & : = & \left( \bigvee^n \left( \bigwedge^n x\right) \right)_i \\
\\
(U_nx)_i &: =& \left( \bigwedge^n \left( \bigvee^n x\right) \right)_i \end{array}$

where
\hspace*{4cm} $\begin{array}{lll} \left( \bigwedge^n x \right)_i & : = & x_{i-n} \wedge x_{i-n+1} \wedge \cdots \wedge x_i\end{array} $

is an {\it erosion} with structural element $W=\{-n,-n-1,...,1,0\}$ and

\hspace*{4cm} $\begin{array}{lll}
\left( \bigvee^n x\right)_i &: = & x_i \vee x_{i+1} \vee \cdots \vee x_{i+n}
\end{array}$

is a {\it dilation}  with {\it structural element} $W'=\{0,1,...,n\}$.  Generalizing the $LU$-operators $L_nU_n$ and $U_nL_n$, call a $LULU$-{\it operator} any composition of the basic smoothers $L_n$ and $U_n$, such as $L_3U_4L_2U_1 U_5$. In particular, each $LULU$-operator is a composition of dilations and erosions, that is, an {\it erosion-dilation cascade} (EDC). More generally, each alternating sequential filter (ASF), which by definition \cite{Heijmans} is a composition of morphological openings and closings with structural elements of increasing size, is a EDC with the extra property of featuring the same number of erosions and dilations. 

We will demonstrate our method on two examples of EDC's - the first in one dimension, the second in two dimensions. This method is considerably refined in the next section.

\textbf{Example 1.} Consider $S: = \bigvee^1 \bigwedge^2 \bigvee^3$. It is a  cascade of an erosion $\bigwedge^2$ and dilations $\bigvee^1, \bigvee^3$ (but not an ASF).
To compute the distribution transfer of $S$, let $X$ be a bi-infinite sequence of independent identically distributed random variables $X_i$. Put
\begin{equation}
Y_i = \left( \bigvee^3X\right)_i, \quad  Z_i : = \left( \bigwedge^2Y\right)_i, \quad  A_i : = \left( \bigvee^1Z\right)_i.\label{13}
\end{equation}
Thus $Y, Z, A$ are again bi-infinite sequences of identically distributed (though dependent) random variables. Let $t\in\mathbb{R}$ and $p=F_X(t)$. Then
\begin{eqnarray*}
\phi_S(p) & =& F_{SX}(t)\ =\ F_A(t)\ =\  P(A_0 \leq t) \\
\\
& = & P(Z_0 \vee Z_1 \leq t) \ = \ P(Z_0 \leq t \quad \mbox{and} \quad Z_1 \leq t)
\end{eqnarray*}

In order to reduce the $Z_i$'s to the $Y_i$'s we switch all $\leq t$ to $> t$ by using exclusion-inclusion (the case $n =2$ in Lemma 7):
\begin{eqnarray*}
P(Z_0, Z_1 \leq t) & =& P(Z_0 \leq t) - P(Z_0 \leq t, Z_1 > t) \\
\\
&= & P(Z_0 \leq t)- \ds\left( \ P (Z_1 > t) - P(Z_1, Z_0 > t) \ \right) \\
\\
&= & 1 - P(Z_0 > t) - P(Z_1 > t) + P(Z_1, Z_0 > t)
\end{eqnarray*}

Since {\it our} $Z_i$'s are identically distributed we have $P(Z_0 > t)  = P(Z_1 > t)$ and hence

\begin{eqnarray*} \phi_S (p) & = & 1 - 2 P(Z_0> t) + P(Z_1, Z_0 > t) \\
\\
&= & 1 - 2 P(Y_{-2} \wedge Y_{-1} \wedge Y_0 > t) + P(Y_{-1} \wedge Y_0 \wedge Y_1, \  Y_{-2} \wedge Y_{-1} \wedge Y_0 > t) \\
\\
& = & 1 - 2P(Y_{-2}, Y_{-1}, Y_0 > t) + P(Y_{-2}, Y_{-1}, Y_0, Y_1 > t)\\
\\
& =& 1-2P (Y_0, Y_1, Y_2 > t) + P(Y_0, Y_1, Y_2, Y_3 > t) \end{eqnarray*}

By the dual of Lemma 7 and because e.g. $P(Y_0, Y_1 \leq t) = P(Y_1, Y_2 \leq t) = P(Y_2, Y_3 \leq t)$ we get

\begin{eqnarray*}
\phi_S(p) &= & 1-2 \ds\left( \ 1-3P(Y_0 \leq t) + 2P(Y_0, Y_1 \leq t ) + P(Y_0, Y_2 \leq t) - P(Y_0, Y_1, Y_2 \leq t) \ \right)\\
\\
 & & + \ds\left( \ 1 - 4P(Y_0 \leq t) + 3P(Y_0, Y_1 \leq t) + 2 P(Y_0, Y_2 \leq t) + P(Y_0, Y_3 \leq t) - 2P(Y_0, Y_1, Y_2 \leq t)  \ \right.\\
\\
& & \ds\left. - 2P(Y_0, Y_1, Y_3 \leq t) + P(Y_0, Y_1, Y_2, Y_3 \leq t) \ \right) \\
\\
& = & 2P(Y_0 \leq t) - P(Y_0, Y_1 \leq t) + P(Y_0, Y_3 \leq t) - 2P(Y_0, Y_1, Y_3 \leq t) + P(Y_0, Y_1, Y_2, Y_3 \leq t)\\
\\
& = & 2P(X_0, X_1, X_2, X_3 \leq t) - P(X_0, X_1, X_2, X_3, X_4 \leq t) + P(X_0, X_1, \cdots, X_6 \leq t) \\
\\
& & - 2P(X_0, X_1, \cdots X_6 \leq t) + P(X_0, X_1, \cdots, X_6 \leq t)\\
\\
& = & 2p^4-p^5+p^7 -2p^7 + p^7\\
\\
& = & 2p^4-p^5
\end{eqnarray*}

\textbf{Example 2.} Let $S$ be an opening on $\R^{\Z \times \Z}$ with defining structural element a $2 \times 2$ square. Let now $X$ be an infinite $2$-dimensional array of independent identically distributed random variables $X_{(i,j)}$ where $(i,j)$ ranges over $\Z \times \Z$. In order to derive the output distribution of $S$ we put

\begin{eqnarray*}
Y_{(i,j)} & : = & X_{(i,j)} \wedge X_{(i-1, j)} \wedge X_{(i,j+1)} \wedge X_{(i-1, j+1)} \\
\\
Z_{(i,j)} & :  = & Y_{(i,j)} \vee Y_{(i+1, ,j)} \vee Y_{(i, j-1)} \vee Y_{(i+1, j-1)}
\end{eqnarray*}

Let $t\in\mathbb{R}$ and $p=F_X(t)$.
The output distribution of $S$ is
\begin{eqnarray*}
\phi_S (p) & = & P(Z_{(0,0)} \leq t) \\
\\
&= & P(Y_{(0,0)}, Y_{(1,0)}, Y_{(0, -1)}, Y_{(1,-1)} \leq t)
\end{eqnarray*}

Following \cite{ButlerMSc}, which introduced that handy notation in the $1$-dimensional case, we abbreviate the latter as
$$((0,0), (1,0), (0,-1), (1,-1))_Y$$
If say $((\ol{(0,0)}, (1,0), \ol{(0,-1)})_Y$ means
$$P(Y_{(1,0)} \leq t, Y_{(0,0)}, Y_{(0,-1)} > t),$$
then it follows from Lemma 7 and from translation invariance (e.g. $(\ol{(0,0)}, \ol{(1,0)}) = (\ol{(0,-1)}, \ol{(1,-1)})$ that
\begin{eqnarray*}
\phi_L (p) & =& ((0,0), (1,0), (0,-1), (1,-1))_Y\\
\\
& =& 1 - 4 \ol{(0,0)}_Y + 2(\ol{(0,0)}, \ol{(1,0)})_Y + 2 (\ol{(0,0)}, \ol{(0,-1)})_Y + (\ol{(1,0)}, \ol{(0,-1)})_Y \\
\\
& & + (\ol{(0,0)}, \ol{(1,-1)})_Y - (\ol{(0,0)}, \ol{(0,-1)}, \ol{(1,-1)} )_Y - (\ol{(0,0)}, \ol{(1,0)}, \ol{(1,-1)})_Y\\
\\
& & - (\ol{(0,0)}, \ol{(0, -1)}, \ol{(1,0)})_Y - (\ol{(1,0)}, \ol{(0,-1)}, \ol{(1,-1)})_Y + (\ol{(0,0)}, \ol{(1,0)}, \ol{(0,-1)}, \ol{(1,-1)})_Y
\end{eqnarray*}

According to the definition of $Y_{(i,j)}$ we e.g. have

\begin{eqnarray*}
(\ol{(0,0)}, \ol{(0,-1)})_Y  &= & ( \ \ol{(0,0)}, \ol{(-1,0)}, \ol{(0,1)}, \ol{(-1,1)}, \ \  \ol{(0,-1)}, \ol{(-1,-1)}, \ol{(0,0)}, \ol{(-1,0)} \ )_X\\
\\
& = & (\ol{(0,0)}, \ol{(-1,0)}, \ol{(0,1)}, \ol{(-1,1)}, \ol{(0,-1)}, \ol{(-1,-1)})_X
\end{eqnarray*}

Putting $q = 1 - p = P(X_{(0,0)} > t)$ the latter contributes a term $q^6$ to

\begin{eqnarray*}
\phi_L(p) & =& 1 - 4q^4 + 2q^6 + 2q^6 + q^7 + q^7 - q^8 - q^8 - q^8 - q^8 + q^9\\
\\
&= & 1 - 4q^4+ 4q^6 + 2q^7 - 4q^8  + q^9
\end{eqnarray*}

\section{Formulas for the distribution transfer of the major
$LULU$-operators}

As it was already done in the preceding Section  it is
often convenient to use the notation $q=1-p$. For example
the output distribution of $M_n$, $L_n$ and $U_n$
given in (\ref{transferMedian}), (\ref{TransLn}) and
(\ref{TransUn}) respectively can be written in the following
shorter form:
\begin{eqnarray}
\phi_{M_n}(p)&=&\sum\limits_{j=n+1}^{2n+1}{2n+1 \choose
j}p^jq^{2n+1-j},\\
\phi_{L_n}(p)&=&1-(n+1)q^{n+1}+nq^{n+2},\label{pTransLn}\\
\phi_{U_n}(p)&=&(n+1)p^{n+1}-np^{n+2}.\label{pTransUn}
\end{eqnarray}

Theorem 11 below deals with the output distribution of $L_nU_n$ and $U_nL_n$. They were first derived
in \cite{Conradie}, but the statement of the theorem was also
independently proved by Butler \cite{ButlerMSc}. In 5.1 we present
a proof using Butler's "expansion calculus". In 5.2 this method is applied to more complicated
situations.

\subsection{The output distribution of the $LU$-operators}

First, observe that instead of winding up with full blown inclusion-exclusion when switching {\it all} inequalities $> t$ to $\leq t$ (dual of Lemma 7), one can be economic and only switch {\it some} inequalities:

\begin{eqnarray}
(\ol{0}, \ol{1}, \cdots, \ol{n})_X & = & (\ol{0}, \cdots \ol{n-1})_X - (\ol{0}, \cdots, \ol{n-1}, n)_X\\
\nonumber \\
& =& (\ol{0}, \cdots, \ol{n-2})_X - (\ol{0}, \cdots, \ol{n-2}, n-1)_X - (\ol{0}, \cdots, \ol{n-1}, n)_X \nonumber \\
& \vdots & \nonumber \\
& = & (\ol{0})_X - (\ol{0}, 1)_X - (\ol{0}, \ol{1}, 2)_X - \cdots - (\ol{0}, \cdots, \ol{n-1}, n)_X \nonumber \\
\nonumber \\
& = & 1 - (0)_X - \ds\sum_{i=0}^{n-1} (\ol{0}, \cdots, \ol{i}, i+1)_X \nonumber
\end{eqnarray}
\begin{lemma}
{[1, Corollary 4]:} Let $X$ be a bi-infinite sequence of random variables. Then
$$(\ol{0}, \ol{1}, \cdots, \ol{n})_X \ = \ 1- [n+1] (0)_X + \ds\sum_{i=0}^{n-1} [n-i] (0, \ol{1}, \cdots,  \ol{i}, i+1)_X$$
\end{lemma}
Note that for $i =0$ we get the summand $n(0,1)_X$.

{\it Proof:}
From
\begin{eqnarray*}
(k+1)_X &= & (k+1, k)_X + (k+1, \ol{k})_X \\
\\
&= & (k+1, k)_X + (k+1, \ol{k}, k-1)_X+ (k+1, \ol{k}, \ol{k-1})_X \ = \cdots \\
\\
& = & (k+1, k)_X + (k+1, \ol{k}, k-1)_X + \cdots + (k+1, \ol{k}, \cdots, \ol{1}, 0)_X + (k+1, \ol{k}, \cdots, \ol{0})_X
\end{eqnarray*}
follows, by translation invariance, that
\begin{eqnarray*}
(\ol{0}, \cdots, \ol{k}, k+1)_X &= & (k+1)_X - (k, k+1)_X - \ds\sum_{i=0}^{k-1} (k-i-1, \ol{k-i}, \cdots, \ol{k}, k+1)_X \\
\\
& =& (0)_X - (0,1)_X  - \ds\sum_{i=0}^{k-1} (0, \ol{1}, \cdots, \ol{i+1}, i+2)_X
\end{eqnarray*}
Using (17) one derives for (say) $n=4$ that
\begin{eqnarray*}
(\ol{0}, \ol{1}, \ol{2}, \ol{3}, \ol{4})_X & = & 1 - (0)_X - \ds\sum_{k=0}^3 (\ol{0}, \cdots, \ol{k}, k+1)_X\\
\\
&= & 1 - (0)_X - \ds\sum_{k=0}^3 \left[ (0)_X - (0,1)_X - \ds\sum_{i=0}^{k-1}(0, \ol{1}, \cdots, \ol{i+1}, i+2)_X \right]\\
\\
& = & 1- 5(0)_X + 4(0,1)_X\\ \\
& &  + (0, \ol{1}, 2)_X \\
\\
& & + (0, \ol{1}, 2)_X  + (0, \ol{1}, \ol{2}, 3)_X \\
\\
& & + (0, \ol{1}, 2)_X + (0, \ol{1}, \ol{2}, 3)_X + (0, \ol{1}, \ol{2}, \ol{3}, 4)_X\\
\\
& =& 1 - 5(0)_X + \ds\sum_{i=0}^3 [4-i](0, \ol{1}, \cdots, \ol{i}, i+1)_X
\end{eqnarray*}
\hfill $\square$

Unsurprisingly, for {\it dependently} distributed random variables $B_i$ certain combinations of $B_i$'s being $\leq t$ and simultaneously other $B_j$'s being $> t$, are impossible, i.e. have probability $0$. More specifically:
\begin{lemma}
{[1, Theorem 10]:} Let $A$ be a bi-infinite identically distributed sequence of random variables and let $B = \bigvee^r A$. Then

$$(0, \ol{1}, \cdots, \ol{n-1}, n)_B = \left\{ \begin{array}{lll} 0  & , & n \leq r+1\\
\\
(0, \cdots, r, \ol{r+1}, \ol{n-1}, n, \cdots, n +r)_A & , & r +1< n < 2r +4 \end{array} \right.$$
\end{lemma}

For instance, for $n =5, r =1$ we have $r+1< n < 2r+4$, and so
$$(0, \ol{1}, \ol{2}, \ol{3}, \ol{4}, 5)_B = (0, 1, \ol{2}, \ol{4}, 5, 6)_A$$
Let us give an ad hoc argument which conveys the spirit of the proof. In view of $B_i = A_i \vee A_{i+1}$ one e.g. has that $B_5 \leq t \Leftrightarrow A_5, A_6 \leq t$. Using inclusion-exclusion we get
\begin{eqnarray*}
(0, 5, \ol{2}, \ol{4}, \ol{1}, \ol{3})_B &= & (0, 5,  \ol{2}, \ol{4})_B - (0, 5, \ol{2}, \ol{4}, 1)_B - (0, 5, \ol{2}, \ol{4}, 3)_B + (0, 5, \ol{2}, \ol{4}, 1, 3)_B\\
\\
& = & P(A_0, A_1, A_5, A_6 \leq t, B_2, B_4> t ) - P(A_0, A_1, A_2, A_5, A_6 \leq t, B_2, B_4> t)\\
\\
& & -P(A_0, A_1, A_3, A_4, A_5, A_6 \leq t, B_2, B_4> t)
 + P(A_0, A_1, \cdots  A_6 \leq t, B_2, B_4 > t)
\end{eqnarray*}
Since $A_4, A_5 \leq t$ is incompatible with $B_4 = A_4 \vee A_5 > t$, the last two terms are $0$. Furthermore, given that $A_5 \leq t$, the statement
$B_4 > t$ amounts to $A_4 > t$. Ditto, given that $A_2 \leq t$, the statement $B_2> t$ amounts to $A_3 > t$. Hence

\begin{eqnarray*}
& & (0,5, \ol{2}, \ol{4}, \ol{1}, \ol{3})_B \\
\\
& =& P(A_0, A_1, A_5, A_6 \leq t, A_4 > t, B_2 > t) - P(A_0, A_1, A_2, A_5, A_6\leq t, A_4 > t, A_3 > t)\\
\\
&= & (  \ P(A_0, A_1, A_5, A_6 \leq t, A_4 > t) - P(A_0, A_1, A_5, A_6 \leq t, A_4 > t, A_2 \leq t, A_3 {\bf \leq} t) \ ) \\
\\
& & - P(A_0, A_1, A_5, A_6 \leq t, A_4 > t, A_2 \leq t, A_3 {\bf >} t)\\
\\
&= & P(A_0, A_1, A_5, A_6 \leq t, A_4 > t) - P(A_0, A_1, A_5, A_6 \leq t, A_4> t, A_2 \leq t)\\
\\
&= & (0, 1, \ol{2}, \ol{4}, 5, 6)_A
\end{eqnarray*}
Dualizing Lemma 9 yields:
\begin{lemma}
{[1, Corollary 11]:} Let $B = \bigwedge^r A$. Then
$$ (\ol{0},1 \cdots, n-1, \ol{n})_B = \left\{ \begin{array}{lll}  0 & , & n\leq r+1 \\
\\
(\ol{0}, \cdots, \ol{r}, r+1, n-1, \ol{n}, \cdots, \ol{n+r})_A & , & r +1< n < 2r +4 \end{array} \right.$$
\end{lemma}

\begin{theorem}\label{theoLuUndistrtransfer}
The distribution transfer functions of $L_nU_n$ and $U_nL_n$ are:
\begin{eqnarray}
&&\phi_{L_nU_n}(p)= p^{n+1} +np^{n+1} q +p^{2n+2}q + \frac{1}{2}
(n-1)(n+2) p^{2n+2}q^2,\\
&&\phi_{U_nL_n}(p)=1-\phi_{L_nU_n}(q) =  1 - q^{n+1}-npq^{n+1} -
pq^{2n+2} - \frac{1}{2}(n-1)(n+2)p^2q^{2n+2}.\hspace{1cm}
\end{eqnarray}
\end{theorem}
%\begin{proof}
%???MARCEL CAN YOU PLEASE WRITE THIS PROOF ALSO INTRODUCING THE
%BASIC NOTATIONS OF THE "EXPANSION CALCULUS"
%\end{proof}
{\it Proof:} Since $L_n U_n = (\bigvee^n \bigwedge^n )(\bigwedge^n \bigvee^n) = \bigvee^n \bigwedge^{2n} \bigvee^n$ we put
$$A= \bigvee^n X, \quad  B = \bigwedge^{2n} A, \quad C = \bigvee^n B$$
 and calculate
\begin{eqnarray*}
\phi_{L_nU_n}(p) &= & P(C_0 \leq t) \  = \ (0)_C \ = \ (0, \cdots, n)_B\\
\\
&= & 1 - [n+1] (\ol{0})_B + \ds\sum_{i=0}^{n-1} [n-i] (\ol{0}, 1, \cdots, i, \ol{i+1} )_B \qquad (\mbox{dual of Lemma 8})\\
\\
&= & 1 -[n+1] (\ol{0}, \cdots, \ol{2n})_A +n (\ol{0}, \cdots, \ol{2n+1})_A + \ds\sum_{i=1}^{n-1} 0 \qquad (\mbox{Lemma 10}, \ r = 2n)\\
\\
& = & 1-[n+1] \left[ 1- [2n+1] (0)_A + \ds\sum_{i=0}^{2n-1} [2n-i] (0, \ol{1}, \cdots, \ol{i}, i+1)_A \right] \\
\\
& & +n \left[ 1 - [2n+2] (0)_A + \ds\sum_{i=0}^{2n} [2n+1-i] (0, \ol{1}, \cdots, \ol{i}, i+1)_A \right] \qquad (\mbox{Lemma 8})\\
\\
& = & [n+1] (0)_A + \ds\sum_{i=0}^{2n} [i-n] (0, \ol{1}, \cdots, \ol{i}, i+1)_A\\
\\
& =& [n+1] (0)_A - n (0,1)_A + \ds\sum_{i=1}^n [i-n] (0, \ol{1}, \cdots, \ol{i}, i+1)_A \\
\\
& & + (0, \ol{1}, \cdots, \ol{n+1}, n+2)_A + \ds\sum_{i=n+2}^{2n} [i-n] (0, \ol{1}, \cdots, \ol{i}, i+1)_A \\
&= & [n+1](0, \cdots, n)_X - n(0, \cdots, n+1)_X +0 + (0, \cdots, n, \ol{n+1}, n+2, \cdots, 2n+2)_X\\
\\
& &  + \ds\sum_{i=n+2}^{2n} [i-n](0, \cdots, n, \ol{n+1}, \ol{i}, i+1, \cdots, i+1+n)_X \qquad (\mbox{Lemma 9})\\
&= & (n+1) p^{n+1} - np^{n+2} + p^{2n+2} q+ \ds\sum_{i=n+2}^{2n} (i-n) p^{2n+2} q^2\\
\\
&= & p^{n+1} + np^{n+1} - np^{n+1} p+p^{2n+2} q+ (2+3+\cdots +n)p^{2n+2}q^2\\
\\
& = & p^{n+1} +np^{n+1} q+p^{2n+2} q+ \frac{1}{2} (n-1) (n+2)p^{2n+2} q^2
\end{eqnarray*}
\hfill $\square$

From (18) it is clear that $p^{n+1}$ is the highest power of $p$ dividing $\phi_{L_nU_n}$. An easy calculation confirms that, as a polynomial in $p$, the right hand side of (19) is $(2n+3)p^2+(\cdots) p^3 + \cdots$. From Corollary 4 hence follows that $L_nU_n$ is lower robust of order $n+1$, but upper robust only of order 2.

\subsection{The output distributions of the $LULU$-operators $C_n$ and $F_n$}

We consider next the specific, mutually dual LULU-operators 
\begin{eqnarray*}
C_n &=& L_nU_nL_{n-1}U_{n-1} \cdots L_1U_1,\\ \\
\mathcal{F}_n &=& U_nL_nU_{n-1}L_{n-1} \cdots U_1L_1.
\end{eqnarray*}
In view of
\begin{eqnarray*}
C_{n-1} &= & {\bigvee}^{n-1} {\bigwedge}^{2n-2} {\bigvee}^{n-1} C_{n-2}\\
\\
C_n &= & {\bigvee}^n {\bigwedge}^{2n} {\bigvee}^n C_{n-1} \\
\\
& = & {\bigvee}^n {\bigwedge}^{2n} {\bigvee}^{2n-1} {\bigwedge}^{2n-2} {\bigvee}^{n-1} C_{n-2}
\end{eqnarray*}
we define the following doubly infinite sequences of identically distributed random variables. Starting with a sequence $X$ of i.i.d. random variables, put
\begin{eqnarray*}
A & : = & {\bigvee}^{n-1} C_{n-2} \  X\\
\\
B & : = & {\bigwedge}^{2n-2} A\\
\\
C' & : = & {\bigvee}^{n-1} B\\
\\
C & : = & {\bigvee}^{2n-1} B \\
\\
D &: = & {\bigwedge}^{2n} C\\
\\
E & :  = & {\bigvee}^n D
\end{eqnarray*}
\begin{theorem}
{[1, Theorem 14]} With $A, B$ as defined above the output distribution $\phi_{C_n}$ of $C_n$ can be computed recursively as follows:
$$\phi_{C_n} \ = \ \phi_{C_{n-1}} + n (G_{2n} - G_{2n-1}),$$
where
\begin{eqnarray*}
G_{2n} & : = & (0, \cdots, 2n-1, \ol{2n}, 2n+1, \cdots, 4n)_B\\
\\
G_{2n-1} & : = & (\ol{0}, \cdots, \ol{2n-2}, 2n-1, \ol{2n}, \cdots, \ol{4n-2} )_A
\end{eqnarray*}
\end{theorem}
{\it Proof:} First, one calculates
\begin{eqnarray}
\phi_{C_{n-1}} & = & (0)_{C'} \ = \ (0, \cdots, n- 1)_B \nonumber \\
\nonumber  \\
&= & 1 - n(\ol{0})_B + \ds\sum_{i=0}^{n-2} [n-1-i] (\ol{0}, 1, \cdots, i, \ol{i+1})_B \quad (\mbox{dual of Lemma} \ 8) \nonumber \\
\nonumber \\
&= & 1- n(\ol{0}, \cdots, \ol{2n-2})_A + [n-1] (\ol{0}, \cdots, \ol{2n-1})_A + \ds\sum_{i=1}^{n-2} 0 \quad (\mbox{Lemma} \ 10, r = 2n-2) \nonumber \\
\nonumber \\
&= & 1 - n (\ol{0}, \cdots, \ol{2n-2})_A + [n-1] (\ol{0}, \cdots, \ol{2n-1})_A
\end{eqnarray}
The expansion of $\phi_{C_n}$ is driven a bit further:
\begin{eqnarray}
\phi_{C_n} & = & (0)_E = (0, \cdots, n)_D \nonumber \\
\nonumber \\
&= & 1 - [n+1] (0)_D + \ds\sum_{i=0}^{n-1} [n-i] (\ol{0}, 1, \cdots, i, \ol{i+1})_D \quad (\mbox{dual of Lemma} \ 8) \nonumber \\
\nonumber \\
&= & 1 - [n+1] (0, \cdots, 2n)_C +n (\ol{0}, \cdots, \ol{2n+1})_C + \ds\sum_{i=1}^{n-1} 0 \quad (\mbox{Lemma} \ 10,r = 2n)\nonumber \\
\nonumber \\
& = & 1 - [n+1] \left[ 1 - [2n+1](0)_C + \ds\sum_{i=0}^{2n-1}[2n-i] (0, \ol{1}, \cdots, \ol{i}, i+1)_C \right] \nonumber \\
\nonumber \\
& & + n \left[ 1 - [2n+2] (0)_C + \ds\sum_{i=0}^{2n} [2n+1 -i] (0, \ol{1}, \cdots, \ol{i}, i+1 )_C \right] \quad (\mbox{Lemma} \ 8) \nonumber \\
\nonumber \\
& = & [n+1] (0)_C - \ds\sum_{i=0}^{2n} [n-i] (0, \ol{1}, \cdots, \ol{i}, i+1)_C \quad (\mbox{easy arithmetic}) \nonumber \\
\nonumber \\
&= & [n+1] (0, \cdots, 2n-1)_B - n (0, \cdots, 2n)_B - \ds\sum_{i=1}^{2n-1} 0 \nonumber \\
\nonumber \\
& & + n (0, \cdots, 2n-1, \ol{2n}, 2n+1, \cdots, 4n)_B \quad (\mbox{Lemma} \ 9, r = 2n-1) \nonumber \\
\nonumber \\
&= & [n+1] (0, \cdots, 2n-1)_B - n(0, \cdots, 2n)_B + n G_{2n}
\end{eqnarray}
This yields
\begin{eqnarray*}
\phi_{C_n} - nG_{2n} & =& [n+1](0, \cdots, 2n-1)_B - n (0, \cdots, 2n)_B \quad (\mbox{by} \ (21)) \\
\\
&= & [n+1] \left[ 1- 2n(\ol{0})_B + \ds\sum_{i=0}^{2n-2} [2n-1-i](\ol{0}, 1, \cdots, i, \ol{i+1})_B \right] \\
\\
& & -n \left[ 1-[2n+1] (\ol{0})_B + \ds\sum_{i=0}^{2n-1} [2n-i]  (\ol{0}, 1, \cdots, i, \ol{i+1})_B \right] \quad (\mbox{dual of Lemma}  \ 8)\\
\\
& = & 1 - n(\ol{0})_B + \ds\sum_{i=0}^{2n-1} [n-1-i](\ol{0}, 1, \cdots, i, \ol{i+1})_B \quad (\mbox{easy arithmetic}) \\
\\
&= &  1- n(\ol{0}, \cdots, \ol{2n-2})_A + [n-1] (\ol{0}, \cdots, \ol{2n-1})_A + \ds\sum_{i=1}^{2n-2} 0\\
\\
& & -n (\ol{0}, \cdots, \ol{2n-2}, 2n-1, \ol{2n}, \cdots, \ol{4n-2})_A \quad (\mbox{Lemma} \ 10, r = 2n-2)\\
\\
&= & \phi_{C_{n-1}} - nG_{2n-1} \qquad (\mbox{by} \ (20))
\end{eqnarray*}
which, upon adding $nG_{2n}$ on both sides, gives the claimed formula for $\phi_{C_{n}}$. \hfill $\square$

\vskip 1cm

As an example, let us compute the output distribution
of $C_2$. From $A = \bigvee^1C_0X = \bigvee^1X$
follows $B = \bigwedge^2A = \bigwedge^2\bigvee X$. Using
expansion calculus the reader may verify that
$$G_{2n} = G_4 = (0,1,2,3,\ol{4}, 5, 6, 7, 8)_B = p^4 q^2[p+p^2q]^2$$
Similarly one gets
$$G_{2n-1} =G_3 = p^2q^2(1-p^2)^2.$$
Therefore
$$\begin{array}{lll}
\phi_{C_2} & =& \phi_{C_1} +2(G_4-G_3) \\
\\
& =& 3p^3+3p^4-9p^5+4p^6 +4p^7-10p^8 + 4p^9 + 8p^{10}-8p^{11} +2p^{12}
\end{array}$$

As to robustness, from the above representation of $\phi_{C_2}$ and by using
Corollary \ref{corRobust} we obtain that $C_2$ is lower robust of
order 3 like $U_2$. Similar to  $L_2U_2$ discussed in 5.1,
the upper robustness of $C_2$ is not inherited from $L_2$.
Indeed, we have
\[
\phi_{C_2}-1=q^2(2p^{10}-4p^9-2p^8+4p^7+4p^4-p^3-3p^2-2p-1)
\]
which implies that  $C_2$
is upper robust only of order 2. However, upper robustness is not constantly 2; these results were obtained from Theorem 12:

\[
\begin{tabular}{|c|c|c|c|c|c|c|}\hline
&&&&&&\\$n$&\ \ 1\ \ &\ \ 2\ \ &\ \ 3\ \ &\ \ 4\ \ &\ \ 5\ \ &\ \ 6\ \ \\&&&&&&\\\hline&&&&&&\\
lower robustness&2&3&4&4&5&6\\&&&&&&\\\hline&&&&&&\\
upper robustness&2&2&3&3&4&4\\&&&&&&\\\hline
\end{tabular}
 \]

We mention that some nice {\it closed formula} for $G_{2n}$ and $G_{2n-1}$ is verified for small $n$ and conjectured to hold universally in [1, Section 4.5.6].

\section{Using exclusion instead of inclusion-exclusion}

As witnessed by 5.2,  one can sometimes exploit symmetry to tame the inherent exponential complexity of inclusion-exclusion. However, without the possibility to clump together many identical terms, the number of summands in Lemma 7  is $2^n$, which is infeasible already for $n = 20$ or so.

In [14] on the other hand, some multi-purpose {\it principle of exclusion} (POE) is employed which had been useful in other situations before. When POE is aimed at calculating the output distribution of a stack filter $S$, a prerequisite is that the stack filter\footnote{More precisely, the positive Boolean function that underlies the stack filters must be given in DNF.} $S$ be given as a disjunction of conjunctions $K_i$, i.e. in {\it disjunctive normal form} (DNF). The POE then begins with the calculation of the set $\mbox{Mod}_1$ of all bitstrings that satisfy $K_1$, then from $\mbox{Mod}_1$ it {\it excludes} all bitstrings that violate $K_2$. This yields $\mbox{Mod}_2 \subseteq \mbox{Mod}_1$, from which all bitstrings are excluded that violate $K_3$, and so on. The feasibility of the POE hinges on the compact representation (using wildcards) of the sets $\mbox{Mod}_i$.

The details being given in [14], here we address the question of how one gets the DNF in the first place. Specifically we consider the frequent case that our stack filter $S$ is a EDC (Section 4) whose structural elements are provided. Let us go in medias res by reworking $S = \bigvee^1 \bigwedge^2 \bigvee^2$ of Example 1:

$$\begin{array}{llll}
(SX)_0 &= & A_0 = Z_0 \vee Z_1 & {\rm (DNF)}\\
\\
&= & (Y_{-2} \wedge Y_{-1} \wedge Y_0) \vee (Y_{-1} \wedge Y_0 \wedge Y_1) & {\rm (blow up)} \\
\\
&  = & Y_0 \wedge Y_{-1} \wedge (Y_{-2} \vee Y_1) & {\rm (get \ CNF)} \\
\\
& = & (X_0 \vee X_1 \vee X_2 \vee X_3) \wedge (X_{-1} \vee X_0 \vee X_1 \vee X_2) & {\rm (blow up)}  \\
\\
 & & \wedge \ \left( \left( X_{-2} \vee X_{-1} \vee X_0 \vee X_1\right) \vee \left( X_1 \vee X_2 \vee X_3 \vee X_4 \right) \right) \\
 \\
 & =& (X_0 \vee X_1 \vee X_2 \vee X_3) \wedge (X_{-1} \vee X_0  \vee X_1 \vee X_2) & {\rm (condense)}\\
 \\
 & & \wedge \ (X_{-2} \vee X_{-1} \vee X_0 \vee X_1 \vee X_2 \vee X_3 \vee X_4)\\
 \\
 &= & (X_0 \vee X_1 \vee X_2 \vee X_3) \wedge (X_{-1} \vee X_0 \vee X_1 \vee X_2) & {\rm (condense \ further)}\\
 \\
 &= & X_0 \vee X_1 \vee X_2 \vee (X_3 \wedge X_{-1}) & {\rm (get \ DNF)}
 \end{array}$$
The last line is the sought DNF of $S$. It is obtained by starting with the DNF $Z_0 \vee Z_1$. This gets ``blown up'' to a DNF in terms of $Y_i$'s (using definition (13) of $Z_0$ and $Z_1$). This DNF needs to be switched\footnote{How one switches between DNF and CNF of a positive Boolean function is a well researched topic which we won't persue here.} to CNF ($=$ conjunctive normal form). This in turn is blown up to a CNF in terms of $X_i$'s. Usually the result can and must be condensed in obvious ways (``condense further'' meant that only the inclusion-minimal index sets carry over). Continuing like this one takes turns switching DNF's with CNF's, and blowing up expressions. This is done as often as there are structural elements. As a ``side product'' the so called {\it rank selection probabilities} $RSP[i]$ are calculated. The latter is defined as the probability that the filter selects the $i$-th smallest pixel in the $w$-element sliding window. For instance here $w=5$ and $RSP[1] = RSP[2]=RSP[3] =0, \ RSP[4] =0.4, \ RSP[5]=0.6$.

The fourth author has written a Mathematica 9.0 program\footnote{It is available upon sending an email to mwild@sun.ac.za} which, given the structural elements of any EDC's (also 2-dimensional), first calculates the DNF of $S$ and from it the output distribution $\phi_S(p)$. (Alternatively the DNF of any stack filter, whether EDC or not, can be fed in directly.) Albeit Wild's algorithm is multi-purpose, it managed to calculate $\phi_{C_n}(p)$ up to $n=5$, and the result agreed with Butler's. Written out as EDC we have $C_5 = \bigvee^5\bigwedge^{10} \bigvee^9 \cdots \bigwedge^2\bigvee$ and $C_5$ has a sliding window of length $61$. The corresponding structural elements $\{0,1,2,3,4,5\}$, $\{0,-1, \cdots, -10\}$, $\{0,1, \cdots, 9\}$ and so forth triggered the calculation of a DNF comprising a  plentiful $12018$ conjunctions (time: $168224$ sec). From this $\phi_{C_5}(p)$ was calculated in $45 069$ sec. Here it is:

$12p^5+7x^6-23p^7+19p^8-130p^9+194p^{10} - 59p^{11} - 142p^{12} +460p^{13} - 787p^{14} + 715p^{15} - 7p^{16} - 1030p^{17}$

$+1959p^{18} - 2216p^{19}  + 208p^{20} + 3711p^{21} - 6748p^{22} +8412p^{23} - 7587p^{24} + 2023p^{25} + 4680p^{26}$

$-7903p^{27} + 8839p^{28} - 13540p^{29} + 30009p^{30}-51715p^{31} + 50159p^{32} - 7686p^{33} - 51417p^{34}$

$+ 78198p^{35} - 50589p^{36} + 6900p^{37} - 7680p^{38} + 56330p^{39} - 86905p^{40} + 43710p^{41} + 49540p^{42}$

$- 114680p^{43} + 103390p^{44} - 40555p^{45}-15370p^{46} +33955p^{47} - 25460p^{48} + 11790p^{49} - 3645p^{50}$

$+ 740p^{51} - 90p^{52} + 5p^{53}$

As to the rank selection probabilities of $C_5$ one has

$RSP[1] = \cdots = RSP[4] =0, \ RSP[5]=0.000002, \ RSP[6] = 0.00001,  \cdots, RSP[37] = 0.04701$,

$RSP[38] = 0.04703, \ RSP[39] = 0.04643, \cdots, RSP[58]=0.00012, \ RSP[59] = RSP[60]$

$= RSP[61]= 0$, the maximum being $RSP[38]$.

\section{Conclusion}

As mentioned in the introduction, concepts of robust smoothers related to ours have been previously considered. To quote from the abstract of [10]:

{\it In this paper we focus on rank selection probabilities (RSPs) as measures of robustness as it is well known that other statistical characterization of stack filters, such as output distributions, breakdown probabilities and output distributional influence functions can be represented in terms of RSPs.}

While we agree with this praise of RSPs we don't share the opinion on page 1642 of [10]:

{\it Efficient spectral algorithms exist for the computation of the selection probabilities of stack filters.}  

The article cited in [10] is [15] which offers Boolean derivatives and weighted Chow parameters but no computational evidence of the feasibility of the proposed intricate method. Similarly [16] which reduces the calculations of the RSPs of stack filters of window size $n$ to the corresponding problem for size $n-1$ offers no computational data; its complexity in theory (and likely in practise) is $O(n!)$. In the same vein no numerical experiments are carried out in [17]. It is evident that none of the approaches [15], [16], [17] (and in fact none that handles the models of a Boolean function one by one) scales up to [14].

\end{document}